\documentclass[fontsize=12pt]{amsart}

\usepackage{amsmath}
\usepackage{amsthm}
\usepackage{amssymb}
\usepackage{mathtools}
\usepackage{mathrsfs}
\usepackage{physics}
\usepackage[margin=1in]{geometry}
\usepackage{scrextend}
\usepackage{tikz-cd}
\usepackage[shortlabels]{enumitem}
\usepackage{colonequals}
\usepackage{comment}
\usepackage{graphicx}

\usepackage[colorlinks=true,hyperindex, linkcolor=magenta, pagebackref=false, citecolor=cyan,pdfpagelabels]{hyperref}

\usepackage{bbm}
\newcommand{\acts}{\; \rotatebox[origin=c]{-90}{$\circlearrowright$} \;}

\setlist[itemize]{noitemsep, topsep=0pt}
\setlist[enumerate]{itemsep=0mm, topsep=0pt}

\makeatletter
\renewenvironment{proof}[1][\proofname]{\par
  \pushQED{\qed}%
  \normalfont \topsep\z@skip 
  \trivlist
  \item[\hskip\labelsep
        \itshape
    #1\@addpunct{.}]\ignorespaces
}{%
  \popQED\endtrivlist\@endpefalse
}
\makeatother

\usepackage[notcolon, notquote]{hanging}

\makeatletter
\def\thm@space@setup{\thm@preskip=10pt
\thm@postskip=0pt}
\makeatother

\newtheorem{Lthm}{Theorem}

\newtheorem{Lconj}[Lthm]{Conjecture}

\newtheorem{theorem}{Theorem}
\numberwithin{theorem}{section}
\newtheorem{corollary}[theorem]{Corollary}
\newtheorem{lemma}[theorem]{Lemma}

\newtheorem{prop}[theorem]{Proposition}

\theoremstyle{definition}
\newtheorem{remark}[theorem]{Remark}

\newtheorem{example}[theorem]{Example}

\newcommand{\Q}{\mathbb{Q}}
\newcommand{\Z}{\mathbb{Z}}

\renewcommand{\O}{\mathcal{O}}

\newcommand{\Alb}{\mathrm{Alb}}

\newcommand{\id}{\mathrm{id}}

\usepackage{pdflscape}
\DeclareMathOperator{\Spec}{\mathrm{Spec}}

\newcommand{\can}{\text{can}}

\DeclareSymbolFont{kplargesymbols}{OMX}{jkp}{m}{n}
\DeclareMathAccent{\wt}{\mathalpha}{kplargesymbols}{"65}
\DeclareMathAccent{\wtt}{\mathord}{largesymbols}{"65}

\newsavebox{\pullbacksymbol}
\sbox\pullbacksymbol{%
\begin{tikzpicture}%
\draw (0,0) -- (1ex,0ex);%
\draw (1ex,0ex) -- (1ex,1ex);%
\end{tikzpicture}}

\newsavebox{\pushoutsymbol}
\sbox\pushoutsymbol{%
\begin{tikzpicture}%
\draw (0,0) -- (-1ex,0ex);%
\draw (-1ex,0ex) -- (-1ex,-1ex);%
\end{tikzpicture}}

\newcommand{\reg}{\mathrm{reg}}

\newcommand{\struct}[1]{\O_{#1}}

\newcommand{\embed}{\hookrightarrow}
\newcommand{\onto}{\twoheadrightarrow}

\newcommand{\fPic}{\mathrm{Pic}}
\newcommand{\cA}{\mathcal{A}}

\usepackage{stmaryrd}

\usepackage{xspace}

\newcommand{\etale}{\'{e}tale\xspace}

\numberwithin{equation}{section}

\usetikzlibrary{arrows.meta}

\newcommand*{\TikzArrow}[1][]{\mathbin{\tikz [baseline=-0.25ex,-latex,#1] \draw [#1] (0pt,0.5ex) -- (1.3em,0.5ex);}}%
\newcommand*{\TikzBiArrow}[1][]{\mathbin{\tikz [baseline=-0.25ex,-latex, #1] \draw [#1] (0pt,0.5ex) -- (1.3em,0.5ex) node[midway, above, label={[label distance=-0.4cm]$\sim$}] {} ;}}
\newcommand{\rat}{\TikzArrow[->, densely dashed]}
\newcommand{\birat}{\TikzBiArrow[->, densely dashed]}
\newcommand{\iso}{\TikzBiArrow[->]}

\thanks{During the preparation of this article, the first author was partially supported by an NSF postdoctoral fellowship DMS-2103099 and the third author was supported by the Research
Foundation Flanders (FWO) Grant no. 1280421N “Topology, birational geometry and
vanishing theorems for complex algebraic varieties”.}


\begin{document}

\title{Nowhere vanishing holomorphic one-forms and fibrations over abelian varieties}

\author{Nathan Chen, Benjamin Church, and Feng Hao}

\maketitle


\section{Introduction}

A remarkable theorem of Popa and Schnell \cite{PS14} shows that any holomorphic one-form on a smooth complex projective variety of general type must vanish at some point. This result was the culmination of many papers \cite{Carrell74, GL87, Zhang97, LZ05, HK05} which explored the interactions between one-forms, Hodge theory, and birational geometry (see also \cite{Wei2020, Villadsen21} for some reformulations and generalizations). One way of rephrasing the theorem of Popa and Schnell is that the presence of nowhere-vanishing one-forms on a variety $X$ bounds its Kodaira dimension $\kappa(X)$. From this perspective, it is natural to ask: given a variety $X$ of intermediate Kodaira dimension, what can we say about its geometry if $X$ carries nowhere-vanishing one-forms?

The goal of this paper is to give a partial answer to this question. For a smooth complex projective variety $X$, consider the following property:
\[
\let\scriptstyle\textstyle \substack{ \text{$X$ admits holomorphic one-forms $\omega_{1}, \ldots, \omega_{g} \in H^{0}(X, \Omega^{1}_{X})$ which are} \\ \text{pointwise linearly independent (PLI), and hence nowhere vanishing.}} \tag{$\ast$}
\]
By \cite[\S4]{PS14}, if $X$ satisfies $(\ast)$ then $\kappa(X) \le \dim{X} - g$. In this paper, we will explore what happens when the Kodaira dimension is maximized. Our first result gives a structure theorem for such varieties when they are minimal:

\begin{Lthm}\label{thm:main}
Let $X$ be a smooth minimal projective variety with $\kappa(X) = \dim X - g$. Then $X$ satisfies $(\ast)$ if and only if $X$ admits a smooth morphism to an abelian variety $A$ of dimension $g$. In this case, $X \cong (A' \times Z) / G$ with a diagonal action, where $A' \to A$ is an isogeny with kernel $G$ and $Z$ is a smooth minimal model of general type.
\end{Lthm}

A related conjecture of Kotschick \cite{Kotschick21} links the existence of nowhere vanishing holomorphic one-forms to the existence of a $C^{\infty}$-fiber bundle structure over $S^{1}$. A finer study of varieties admitting nowhere vanishing holomorphic one-forms in the context of Kotschick's conjecture was initiated by Schreieder \cite{Schreieder21}, where he classifies smooth projective surfaces carrying such forms (see also \cite{Kotschick21} for a different argument and \cite{DHL21, SY22} for related work). In subsequent joint work with the third author \cite[Thm. 1.3]{HS21(1)}, they precisely describe all threefolds which admit nowhere vanishing one-forms. In low dimensions, these results particularly cover Theorem~\ref{thm:main} and the following Theorem~\ref{thm:KodairaDimFiniteCoverProduct}, which removes the minimality assumption. The ``$g=1$'' case of Theorem~\ref{thm:main} was first proved by the third author \cite[Thm. 1.1]{Hao23}.

In our approach, the Kodaira dimension assumption is used to produce some map to a lower dimensional abelian variety in the first place (see Lemma~\ref{non_contraction}). However, we propose that the first conclusion of Theorem~\ref{thm:main} should hold without this assumption on $X$:

\begin{Lconj}
Let $X$ be a smooth minimal model and let $g$ be the \textit{maximal} number of PLI one-forms that $X$ carries. Then there exists a smooth morphism to a $g$-dimensional abelian variety.
\end{Lconj}

\noindent Our conjecture is similar in spirit to \cite[Conj.~1.7]{HS21(1)}, although our hope is that the dimension of the abelian variety is more precisely explained by the number of PLI one-forms. As an application of these ideas, we also verify additional cases of a conjecture of the third author \cite[Conj.~1.5]{Hao23}.

\begin{Lthm}\label{thm:smooth_map_to_simpleAV}
Let $f : X \to A$ be a morphism from a smooth minimal projective variety $X$ to a simple abelian variety $A$. If $\kappa(X) \ge \dim{X} - \dim{A}$, then the following are equivalent:
\begin{enumerate}[label={\upshape(\arabic*)}]
\item there exists a holomorphic one-form $\omega \in H^0(A, \Omega^1_A)$ such that $f^{\ast}\omega$ is nowhere vanishing;
\item $f : X \to A$ is smooth.
\end{enumerate}
Furthermore, if this holds then $\kappa(X) = \dim{X} - \dim{A}$ and $f$ is isotrivial with general type fibers.
\end{Lthm}

From condition (2), both the equality and the fact that $f$ is isotrivial follow from known work. Indeed, if we assume that $f \colon X \rightarrow A$ is smooth, a recent result of Meng-Popa \cite{MP21} shows that $\kappa(X) = \kappa(F)$. Together with the Kodaira dimension assumption in Theorem~\ref{thm:smooth_map_to_simpleAV} and the trivial inequality $\dim{F} \geq \kappa(F)$, this implies that $\kappa(X) = \dim{X} - \dim{A} = \dim{F}$ and so $F$ is of general type. Since the $C_{n.m}^{+}$ conjecture is known when the fibers are of general type by \cite{Kollar87}, it follows that $f$ is birationally isotrivial (hence isotrivial by the minimality assumption). For this last part, see also \cite[Cor. 3.2]{PS14}.

For a variety $X$ which is not necessarily minimal, we give a much stronger generalization of \cite[Thm.~1.3]{Hao23}, both in terms of the number of PLI one-forms and in terms of the geometric consequences for $X$.

\begin{Lthm}\label{thm:KodairaDimFiniteCoverProduct}
Let $X$ be an $n$-dimensional smooth projective variety with $\kappa(X) = n - g$, and let $f : X \to A$ be a morphism to an abelian variety $A$. Suppose there exist one-forms $\omega_1, \dots, \omega_g \in H^0(A, \Omega^1_A)$ such that $f^* \omega_1, \dots, f^* \omega_g$ are PLI on $X$. Then there exists a $g$-dimensional quotient $q : A \to B$ and a birational map $X \birat (B' \times Z)/G$ making the diagram
\begin{center}
    \begin{tikzcd}
    X \arrow[d, dashed] \arrow[r, "f"] & A \arrow[d, "q"] 
    \\
    (B' \times Z)/G \arrow[r] & B 
    \end{tikzcd}
\end{center}
commute, where $B' \to B$ is an isogeny with kernel $G$ and $Z$ is a smooth variety of general type with a (not necessarily free) $G$-action.
\end{Lthm}

Note that there are significant differences between this paper and the earlier works, e.g., \cite{HS21(1)} and \cite{Hao23}. For instance, we are able to avoid a careful analysis of degenerate fibers appearing in the Iitaka fibration by applying rigidity results on torsors over abelian schemes (which are recorded in \S3) and leveraging the $G$-equivariant MMP. Using these methods, we are able to bootstrap from birational results obtained entirely over the smooth locus of the Iitaka fibration to the existence of prescribed $G$-covers over the original variety. This also has the advantage of strengthening our conclusions. Instead of arbitrary quasi-\etale covers, we obtain \etale $G$-covers for $G$ a quotient of $\Z^{2g}$. From a different perspective, we would like to also mention that there are results about classifying varieties that carry nowhere vanishing sections of the tangent bundle \cite{AMN12}.

\textbf{Conventions.} A \textit{variety} is an integral separated scheme of finite type over an algebraically closed field $k$ of characteristic zero. A \textit{minimal model} is a projective variety $X$ with terminal $\Q$-factorial singularities such that $K_{X}$ is nef. We say a minimal model is \textit{good} if $n K_X$ is base-point-free for some $n > 0$.

\textbf{Acknowledgements.} We would like to thank Sean Cotner, Eleny Ionel, Rafe Mazzeo, Johan de Jong, James M\textsuperscript{c}Kernan,  Mihnea Popa, Christian Schnell, Stefan Schreieder, Ravi Vakil, and Chenyang Xu for the many insightful discussions and suggestions.

\section{PLI holomorphic one-forms}

Let $X$ be a smooth complex projective variety with Kodaira dimension $\kappa(X) \geq 0$, and suppose that $X$ satisfies condition $(\ast)$. This is equivalent to the existence of a vector sub-bundle $\O_X^{\oplus g} \subseteq \Omega_X^{1}$, or alternatively a subspace
\[ W \subset H^0(X, \Omega_X^1) \]
whose nonzero elements are all nowhere vanishing holomorphic one-forms. We will pass between these characterizations without comment. An alternative form of our main question asks if such subbundles $\struct{X}^{\oplus g} \subset \Omega_X$ are \textit{explained} by a smooth morphism $X \to A$ to an abelian variety. From the existence of a smooth morphism to an abelian variety of dimension $g$, one can recover some set of PLI one-forms but these may not be the ones we started with in ($\ast$). We use the terminology \textit{explained by} rather than saying the subbundle arises from the morphism $f : X \to A$ since the stronger property almost never holds, as the following example demonstrates.

\begin{example}
Let $C$ be a genus $g \ge 2$ whose Jacobian has only simple isogeny factors of dimension $\ge 2$ and let $E$ be an elliptic curve. Consider the product $X = E \times C$ with its projection maps $\pi_{i}$. Choose one-forms $\omega_{E} \in H^{0}(E, \Omega_{E}^{1})$ and $\omega_{C} \in H^{0}(C, \Omega_{C}^{1})$. First, the product $X$ carries two nowhere-vanishing one-forms which are \textit{globally independent} but not PLI: $\pi_{1}^{\ast}\omega_{E}$ and $\pi_{1}^{\ast}\omega_{E} + \pi_{2}^{\ast}\omega_{C}$ so it is important to consider only PLI forms. Furthermore, the generic non-vanishing one-form on $X$ takes the form $\pi_{1}^{\ast}\omega_{E} + \pi_{2}^{\ast}\omega_{C}$ with $\omega_C \neq 0$. However, any morphism $X \to E'$ to an elliptic curve factors through $\pi_1$ since any morphism $C \to E'$ is constant by assumption. Therefore, the only one-form which is the pullback along a smooth map to an elliptic curve is $\pi_1^{\ast} \omega_{E}$. In a looser sense, we will say that the entire Zariski open of $H^0(X, \Omega_X)$ consisting of nowhere vanishing forms is \textit{explained} by the smooth map $\pi_1 : X \to E$. We aim to explain PLI forms on higher-dimensional varieties in a similar sense. 
\end{example}

Our goal now is to produce a candidate map from $X$ to an abelian variety of dimension $g$. Let $\phi : X \rat X^{\can}$ be the Iitaka fibration of $X$, consider a desingularization $S$ of $X^{\can}$, and let $X' \to X$ be a smooth resolution of $X \rat S$. The Albanese morphism $X \to \Alb_{X}$ induces a rational map $X' \to \Alb_X$, which is everywhere defined since $\Alb_X$ is an abelian variety. These fit into the diagram below:
\begin{center}
\begin{tikzcd}[row sep=small]
& F \arrow[rr] \arrow[d] & & \Alb_{F} \arrow[d]
\\
X \arrow[dd, dashed] & X' \arrow[l] \arrow[rr] \arrow[dd] & & \Alb_{X} \arrow[dd, two heads] \arrow[dl, two heads]
\\
&  & Q_X \arrow[rd, two heads] &
\\
X^{\can} & S \arrow[l] \arrow[ru, dashed] \arrow[rr] & & \Alb_S
\end{tikzcd}
\end{center}

\noindent The two left-facing morphisms are birational, so we will freely use $\Alb_{X} \cong \Alb_{X'}$. In the diagram, $F$ denotes a general fiber of $X' \rightarrow S$ and $Q_X$ is the cokernel of $\Alb_F \to \Alb_X$. This gives the third column of induced morphisms on Albanese varieties. Since the fiber $F$ is contracted in $X' \to Q_X$ by definition, it factors birationally through $X' \to S$ by rigidity. 

\begin{lemma}[{c.f. \cite[Proof of Conjecture 1.2]{PS14}}]
Following the notation above, the inequalities hold:
\begin{equation}\label{eq:1}
    \dim{F} \ge \dim{\Alb_X} - \dim{Q_X} \ge \dim{W}.
\end{equation}
\end{lemma}

\begin{proof}
The first inequality follows from the fact that $\dim{F} \ge \dim{\Alb_F}$ \cite[Thm.~1]{Kawamata81} and $\dim{Q_X} \ge \dim{\Alb_X} - \dim{\Alb_F}$ (by definition of $Q_X$). For the second inequality, note that the composition morphism $f \colon X \to \Alb_{X} \to Q_X$ birationally factors through the Iitaka fibration on $X$. Therefore, we may apply \cite[Thm.~2.1]{PS14} to show that
\[ W \cap f^{*} H^{0}(Q_X, \Omega_{Q_X}^1) = \{ 0 \}, \]
which implies that $\dim W + \dim Q_X \leq \Alb_{X}$. Note that $f^{\ast}$ is injective since $\Alb_X \to Q_X$ is smooth and surjective.
\end{proof}

Given the above inequalities, it is natural to study the extremal case: when the Kodiara dimension is maximized with respect to property $(\ast)$. In this case, we can produce a map from $X$ to an abelian variety.

\begin{theorem}\label{non_contraction}
If $\kappa(X) = \dim{X} - \dim{W}$, then $\dim{F} = \dim{W}$ so the above are all equalities. This has the following consequences:
\begin{enumerate}[label={\upshape(\roman*)}]
\item $F \to \Alb_F$ is a birational morphism,
\item $\Alb_F \to \Alb_X$ is finite \etale onto its image,
\item $F \to \Alb_X$ is a generically finite map onto a translate of an abelian subvariety of $\Alb_{X}$,
\item $X$ admits a surjective morphism to an abelian variety of dimension $g = \dim{W}$.
\end{enumerate}
\end{theorem}

\begin{proof}
Suppose $\kappa(X) = \dim{X} - \dim{W}$, or equivalently $\dim{F} = \dim{W}$. In this case, all of the inequalities in \eqref{eq:1} become equalities, $\kappa(F) = 0$, and $\dim{F} = \dim{\Alb_F}$ so (i) holds by \cite[Thm.~1]{Kawamata81}. (ii) follows from the fact that $\dim{Q_X} + \dim{\Alb_F} = \dim{\Alb_X}$, and (iii) is the composition of (i) and (ii).

By varying the fiber $F$, rigidity implies that the images of $F \to \Alb_X$ are translates of a fixed abelian subvariety, say $B \subset \Alb_X$. Fixing a polarization on $A$ and dualizing yields a morphism
\[ q : A \to A^\vee \onto B^\vee \]
such that $B \subset A$ maps surjectively onto $B^\vee$. We claim that $X \to B^\vee$ is surjective, which follows from the fact that $F \to B$ is surjective.
\end{proof}


In the rest of this section, we present two useful lemmas that follow from existing results in MMP.

\begin{lemma}\label{lemma:existence-gmm}
    Let $X$ be a smooth projective variety satisfying $(\ast)$ with $\kappa(X) = \dim X - g$. Then $X$ admits a good minimal model.
\end{lemma}

\begin{proof}
    Lemma~\ref{non_contraction} implies that in this case, the Iitaka fibration is birationally fibered in abelian varieties, so we may apply \cite[Thm.~4.4]{Lai11}.
\end{proof}

\begin{lemma} \label{lemma:minimal_birat_to_abelian_var}
Let $X$ be a variety with canonical singularities and $K_X$ numerically trivial. If $f : X \birat A$ is a birational map to an abelian variety then $f$ extends to an isomorphism.
\end{lemma}

\begin{proof}
Since $X$ is normal and canonical, by \cite{BCHM10} there exists a crepant terminalization morphism $\tau : X' \to X$ and thus $K_{X'} = \tau^* K_{X}$. But $K_{X}$ is numerically trivial and in particular nef, so $X'$ is a minimal model birational to the abelian variety $A$. The rational map $X' \rat A$ is then a sequence of flops and hence an isomorphism since any flopping locus on $A$ must be trivial. We claim that the maps $X' \to X \rat A$ are isomorphisms. Indeed, $X$ has rational singularities so $X \to A$ is a morphism. Since both maps are finite with connected fibers, they must both be isomorphisms.
\end{proof}

\section{Torsors for Abelian Schemes}

In this section, we present some results about the global structure of torsors over constant abelian schemes and isogenies between isotrivial abelian schemes. We expect this material to be well known to experts but do not know of a reference that covers the exact results we require. In this section only, we work over an arbitrary field $k$ for additional clarity. Given an abelian variety $A$ over $k$ and a $k$-scheme $S$, we will write $A_{S} \colonequals A \times_{k} S$ for the abelian $S$-scheme.

\begin{lemma}\label{lemma:TorsorOverPic^0}
Let $\psi : X \to S$ be a smooth flat morphism whose geometric fibers are abelian varieties. Then $X \to S$ is a torsor over the abelian scheme $\cA \colonequals (\fPic_{X/S}^0)^\vee$.
\end{lemma}

\begin{proof}
Since $\psi$ is smooth, after an \etale cover $S' \to S$, it admits a section and hence becomes an abelian scheme isomorphic to $\cA$ over $S'$. This isomorphism depends on the choice of section and hence is noncanonical. However, the induced translation action $\cA_{S'} \acts X_{S'}$ is \textit{canonical} since it does not depend on the choice of section. Indeed, this is a restatement of the observation that the translation action of $A^{\vee \vee} \acts A$ does not depend on the choice of basepoint of $A$ (where the action is defined with respect to the canonical basepoint of $A^{\vee \vee}$). The uniqueness of this action (e.g. the two induced actions over $S' \times_S S'$ agree) implies that it descends to an action over $S$. Checking that it is a torsor can also be done \etale-locally whence it is obvious because then $X \to S$ is isomorphic to $\cA$ \etale-locally.
\end{proof}

\begin{lemma} \label{lemma:isotrivial_constant}
Let $k$ be a field and $S$ a $k$-scheme that has a $k$-point. Let $f : \cA \to A_S$ be a finite homomorphism of abelian schemes, where $A$ is an abelian variety over $k$. If $\deg{f}$ is invertible in $k$, then $\cA \iso A'_S$ over $S$ where $A' \to A$ is a $k$-isogeny of abelian varieties. 
\end{lemma}

\begin{example}
This is false without the assumption that $S$ has a $k$-point. For instance, if $S = \Spec{k'}$ for a finite extension $k'/k$, then there can exist abelian varieties $\cA$ defined over $k'$ (but not defined over $k$) which map to abelian varieties defined over $k$.  
\end{example}

\begin{example}
This is additionally false if $k$ has characteristic $p$ and $p \mid \deg{f}$, since in this case $G$ may not be \etale over $S$. Indeed, there can be nonisotrivial $\cA$ becuase of the existence of continuous moduli for subgroups of $A[p^n]$.
\end{example}

For the proof, we need the following well-known rigidity lemma for finite \etale subschemes.

\begin{lemma} \label{lemma:rigidity_finite_etale}
Let $X \to S$ be a finite \etale $S$-scheme and $Z, Z' \subset X$ closed subschemes flat and finitely presented over $S$. If $S$ is connected and for some $s \in S$ we have $Z_s = Z'_s$ as subschemes of $X_s$ then $Z = Z'$ as subschemes of $X$.
\end{lemma}

\newcommand{\Hilb}{\mathrm{Hilb}}

\begin{proof}
We will show that the locus of points $s \in S$ where $Z_{s} = Z'_{s}$ is clopen. Being closed is clear so we must show that it is open. By reducing to the Noetherian case, then reducing to complete local rings of $S$, it suffices to show that subschemes $Z \subset X_A$ lift uniquely under extensions of Artin rings $A' \onto A$. The latter statement follows directly from topological invariance of the \etale site. Instead of the series of reductions, one may also argue that the lifting criterion implies that $\Hilb_{X/S}$ is represented by a finite \etale $S$-scheme. Then $Z, Z'$ correspond to a pair of sections $S \rightrightarrows \Hilb_{X/S}$ which by assumption agree at $s \in S$. Since $\Hilb_{X/S} \to S$ is \etale and separated, and the diagonal $\Delta_{\Hilb_{X/S}/S}$ is a clopen embedding, by base change the locus where the maps $S \rightrightarrows \Hilb_{X/S}$ agree is clopen. 
\end{proof}

\begin{proof}[Proof of Lemma~\ref{lemma:isotrivial_constant}]
Consider the dual isogeny $A^\vee_S\to \cA^\vee$. The kernel $G \subset A_S$ is a finite \etale subgroup scheme over $S$. If we let $s \in S(k)$, then the fiber over $s$ is an isogeny $A^\vee \to \cA_s^\vee$ whose kernel is $G_s \subset A^\vee$. Since both $G$ and $G_s \times S$ are closed subgroups of a finite \etale $S$-scheme $(A^\vee_S)[\deg{f^\vee}]$ and they agree at $s$ by definition, from Lemma~\ref{lemma:rigidity_finite_etale} we conclude that $G = G_s \times_k S$ and hence $\cA \cong (A/G_s^\vee)_S$.
\end{proof}

For the remainder of this section, we apply this formalism to show that certain families of abelian varieties arise as the quotient of a constant family by a finite group acting diagonally. The role of diagonal group actions in similar extremal Kodaira dimensional classification results was demonstrated in \cite{HS21(2)}, which utilizes a structure theory for finite quotients of products with elliptic curves. In \cite[\S3]{HS21(2)} the authors use a reparameterization technique to obtain the diagonal action; our approach differs but the diagonal nature of the action is similarly crucial to our applications.

\begin{prop}\label{prop:AntiDiagonalQuotient}
Let $k$ be a field of characteristic zero. Let $\psi : X \to S$ be a smooth flat morphism whose geometric fibers are abelian varieties and $s \in S(k)$. Let $f : X \to A \times_k S$ be a finite surjective $S$-morphism and let $S' \colonequals f^{-1}(0 \times A) \overset{\iota}{\hookrightarrow} X$. Set $A' = X_s$ and $G = \ker{(A' \to A)}$. Then $S' \to S$ is a $G$-torsor and there is an $S$-isomorphism compatible with $f$:
\[ X \iso (A' \times_k S')/G \]
where $G$ acts on $A' \times_k S'$ via the action: $g \cdot (a', s') = (g^{-1} \cdot a', g \cdot s')$.
\end{prop}

\begin{remark}
Note that finiteness and surjectivity of $f$ can be checked on fibers \cite[Chapter III \S4, Proposition 4.6.7(i)]{EGA3}. Thus, to apply this result it will suffice to show that the induced maps on fibers are dominant generically finite maps of abelian varieties, so automatically isogenies. 
\end{remark}

\begin{proof}[Proof of Proposition~\ref{prop:AntiDiagonalQuotient}]
By Lemma~\ref{lemma:TorsorOverPic^0}, $\psi : X \to S$ is a torsor over the abelian scheme $\cA = (\fPic^0_{X/S})^\vee$. Applying the double dual functor $\fPic^{0}(\fPic^{0}_{-/S})$ to $f$ gives an isogeny $\cA \to A \times_k S$, so by Lemma~\ref{lemma:isotrivial_constant} we have $\cA \cong A'_S$. Then $G \times_k S = \ker{(\cA \to A_S)}$ acts on the fibers of $X \to A_S$ over $S$ since $X \to A_S$ is equivariant with respect to $\cA \to A_S$. In particular, $G \acts S'$ over $S$. With this action, $S' \to S$ is a $G$-torsor. Indeed, being a $G$-torsor can be checked after an \etale cover, so we pass to a cover on which $X$ and $\cA$ become isomorphic. Now, the $\cA_{S'}$-torsor $X \times_S S' \to S'$ has a section $\sigma = (\iota, \id)$ and hence is trivial. Explicitly, there is an $\cA$-equivariant $S'$-isomorphism,
\[ \Phi : \cA \times_S S' \iso X \times_S S' \]
which is defined by $\Phi = (\rho \circ (\id \times \iota), \pi_2)$ where $\rho : \cA \times_S X \to X$ is the action. By definition, this is equivariant for the $\cA$-action on the left. However, $\cA \times_S S'$ also admits $G$-action ``on the right'' arising from the $G$-action on the $S'$ factor. The isomorphism $\Phi$ is also equivariant for this second action in the following sense. It is most convenient to describe these actions in terms of the functor of points. In this notation, for an $S$-scheme $T$, the action ``on the right'' $G \acts (\cA \times_S S')$ is transformed under $\Phi$ as follows,
\[ \Phi : (a, g \cdot s') \mapsto (a g \cdot s', g \cdot s') = (g \cdot (a \cdot s'), g \cdot s') \]
Using commutativity in an essential way, we see that $\Phi$ intertwines the ``right'' $G$-action on $\cA \times_S S'$ with the diagonal action on $X \times_S S'$. Therefore, $\Phi$ intertwines the ``right'' $G$-action on $X \times_S S'$ given by acting on $S'$, with the anti-diagonal action on $\cA \times_S S'$ meaning that
\[ \Phi :  (g^{-1} \cdot a, g \cdot s') \mapsto (a \cdot s', g \cdot s').  \]
Observe that for this ``right'' action, $(X \times_S S') / G = X \times_S (X'/G) = X$ because $S' \to S$ is a $G$-torsor. Therefore, we get an isomorphism
\[ (A' \times_k S') / G = (\cA \times_S S') / G \xrightarrow{\Phi} (X \times_S S')/G = X \]
with $G$ acting on $A' \times_k S'$ via the desired action.
\end{proof}

\section{Proofs of the main results}

The goal of this section is to prove a generalization of Theorem~\ref{thm:main} (see Theorem~\ref{thm:factoringthroughAV}). Given a morphism $f : X \rightarrow Y$, we say that $f$ \textit{contracts} a subvariety $F \subset X$ if the map $F \to f(F)$ drops in dimension. The main technical result that we need concerns varieties equipped with a fibration as well as a morphism $f : X \to A$ to an abelian variety not contracting the fiber:

\begin{theorem}\label{thm:GeneralBirDecomposition}
Let $X$ be a smooth variety equipped with maps $\psi : X \to S$ and $f : X \to A$ such that $\psi$ is proper, $f$ does not contract the general fiber of $\psi$, and $K_X = \psi^* L$. Then there exists a quotient $q : A \to B$ to an abelian variety $B$ of dimension equal to the relative dimension of $\psi$ and a birational map $X \birat (B' \times Z)/G$ making the diagram
\begin{center}
    \begin{tikzcd}
    X \arrow[d, dashed] \arrow[r, "f"] & A \arrow[d, "q"] 
    \\
    (B' \times Z)/G \arrow[r] & B 
    \end{tikzcd}
\end{center}
commute. Here, $B' \to B$ is an isogeny with kernel $G$ and $Z$ is a smooth variety with a $G$-action.
\end{theorem}

\begin{proof}
We will first show that every smooth fiber over the regular locus $S^{\text{reg}}$ is isomorphic to a fixed abelian variety and maps to $A$ as an \etale cover onto a translate of a fixed abelian subvariety $B_0 \subset A$. After Stein factorizing, we may assume that $\psi$ has geometrically connected fibers and $S$ is normal.

Let $F$ be a general fiber of $\psi$. Since $f$ does not contract $F$, we claim that no fiber of $\psi$ can be contracted by rigidity, and furthermore, the image of each fiber of $\psi$ under the map $f$ is a translate of a fixed abelian subvariety. This follows from a standard argument using the fact that an abelian variety admits at most countably many abelian subvarieties (c.f. \cite[Thm. 13]{Kawamata81}). Here we present an alternative approach. Using adjunction, $\kappa(F) = 0$ so by \cite[Thm.~1]{Kawamata81} the map $F \to \Alb_F$ is surjective and hence the image of $F$ in $A$ is a translate of an abelian subvariety, say $B_0 \subset A$ where $\dim B_{0} = \dim F$. The composition $X \to A \to A / B_0$ then contracts $F$, so by rigidity it factors:
\begin{center}
\begin{tikzcd}
X \arrow[r, "f"] \arrow[d, swap, "\psi"] & A \arrow[r] & A / B_0 
\\
S \arrow[rru, dashed]
\end{tikzcd}
\end{center}
Restricting to the regular locus of $S$, the rational map $S^{\text{reg}} \to A / B_0$ is a morphism since $A / B_0$ is an abelian variety. Hence, $X \to A \to A / B_0$ contracts every fiber over the regular locus. By uppersemicontinuity of the fiber dimension, all fibers of $\psi$ map surjectively onto translates of $B_0$. For every smooth fiber $F$ we know by adjunction that $K_F = 0$ so $F$ is minimal. From our assumption that $f$ does not contract $F$ and \cite[Thm.~1]{Kawamata81}, the composition $F \to \Alb_F \to A$ is generically finite onto its image and hence $F \to \Alb_F$ is birational. Since $F$ is smooth and minimal, by Lemma~\ref{lemma:minimal_birat_to_abelian_var} the map $F \to \Alb_F$ is an isomorphism so $F \to A$ as an \etale cover of its image, which is a translate of $B_0$. In particular, every abelian variety fiber of $\psi$ is isogenous to $B_0$. By countability of isogeny classes, $\psi$ is isotrivial over $S^{\reg}$.
\par
Let $U \subset S$ be an open such that $\psi_U$ is smooth; note that $U$ exists by generic smoothness since $X$ is smooth. However, the map $f : X_U \to A \times U$ may not be surjective. To rectify this, we apply the following trick from the proof of \cite[Thm.~5.10]{HS21(1)}. First take the dual morphism $A^{\vee} \rightarrow B_0^\vee$ of the inclusion. To get the required morphism, we choose an ample class on $A$ which induces a polarization $A \to A^\vee$. Hence the composition
\[ X_U \to A \rightarrow A^{\vee} \rightarrow B_{0}^\vee \]
satisfies the required hypotheses to apply Theorem~\ref{prop:AntiDiagonalQuotient} since $B_0 \to B_0^{\vee}$ is an isogeny. Setting $B = B_0^\vee$, we then get a finite surjective map $X_U \to A \times U$ over $U$ and $X_U \to U$ is a smooth proper morphism whose fibers are abelian varieties. Then Proposition~\ref{prop:AntiDiagonalQuotient} proves that $X_U \cong (B' \times U') / G$ for an isogeny $B' \to B$ with kernel $G$.
\par
Now we choose a smooth equivariant compactification $U' \embed Z$, meaning $Z$ is smooth and projective with a $G$-action and $U' \embed Z$ is an equivariant open embeding. This relies on an equivariant version of Nagata compactification, which holds only for finite group actions. Given a finite group action $G \acts X$, consider the scheme-theoretic image of $X \to \prod_{g \in G} \overline{X}$ given by applying $g$. Then one can embed $X \embed \overline{X}$ for any given Nagata compactification \cite{Nagata63} and resolve equivariantly using the existence of a functorial resolution of singularities \cite[Prop.~3.9.1]{Kollar09}. From the equivariant compactification above, we have
\[ X \birat X_U \iso (B' \times U')/G \embed (B' \times Z)/G \]
where $(B' \times Z)/G$ is a smooth projective variety since $Z$ is smooth projective and $G \acts A$ freely.
\end{proof}

\begin{corollary}\label{corollary:G-cover}
In the above case, $X$ admits a finite \etale $G$-cover $X' \to X$ such that $X'$ is equivariantly birational to $A' \times Z$, where $Z$ is general type and $G$ is a quotient of $\Z^{2g}$. 
\end{corollary}

\begin{proof}
The birational equivalence $X \birat (A' \times Z)/G$ of smooth projective varieties induces an equivalence of \etale fundamental groups. Therefore, the finite \etale cover $A' \times Z \to (A' \times Z)/G$ extends to a birationally equivalent finite \etale cover $X' \to X$.
\end{proof}

\begin{remark}
This proves Theorem \ref{thm:FiniteCoverProduct} below in the case that $X$ is a smooth minimal model and $\psi$ is the Iitaka fibration. We will need to do a little more work to extract the general case. 
\end{remark}

\begin{theorem}\label{thm:FiniteCoverProduct}
Let $X$ be an $n$-dimensional smooth projective variety. Suppose there exists a morphism $f : X \to A$ to an abelian variety $A$ such that the general fiber of the Iitaka fibration is not contracted by $f$. Then there exists a $(\dim X - \kappa(X))$-dimensional quotient $q : A \to B$ and a birational map $X \birat (B' \times Z)/G$ making the diagram
\begin{center}
    \begin{tikzcd}
    X \arrow[d, dashed] \arrow[r, "f"] & A \arrow[d, "q"] 
    \\
    (B' \times Z)/G \arrow[r] & B 
    \end{tikzcd}
\end{center}
commute. Here, $B' \to B$ is an isogeny with kernel $G$ and $Z$ is a smooth variety of general type with a (not necessarily free) $G$-action.
\end{theorem}

\begin{proof}
By the same proof as in Lemma~\ref{lemma:existence-gmm}, our assumption on $f$ implies that the general fiber $F$ of the Iitaka fibration $\phi : X \rat S$ is birational to an abelian subvariety $B_0 \subset A$. So there exists a good minimal model $X^{\min}$ of $X$, which implies that the Iitaka fibration $\phi : X^{\min} \to S$ is a morphism. We would like to apply Theorem~\ref{thm:GeneralBirDecomposition} directly to $\psi$; however $X^{\min}$ may not be smooth. Note that $X^{\min}$ is canonical, so we know that the general fiber $F$ of $\phi$ is reduced and has canonical singularities \cite[Lem.~5.17]{KM98}. Instead of directly imitating the proof of Theorem \ref{thm:GeneralBirDecomposition}, we will directly analyze the general fiber $F$ and show that it is an abelian variety. Applying the Stein factorization to $F \to \wt{B}_0 \to B_0$, we have a normal variety $\wt{B}_0$ with $Y_0 \to \wt{B}_0$ birational. Hence $\kappa(\wt{B}_0) = 0$ so by \cite{KV80} (c.f. \cite[Thm. 4]{Kawamata81}) the finite map $\wt{B}_0 \to B$ is an \etale map and $\wt{B}_0$ is an abelian variety. By Lemma~\ref{lemma:minimal_birat_to_abelian_var}, we conclude that $F \iso \wt{B}_0$ is an isomorphism.

Let $U \subset S^{\reg}$ be an open subset such that $\phi_{U} : X^{\min}_U \to U$ is flat and every fiber has at worst canonical singularities. Then every fiber is an abelian variety by the above argument, so $\phi_{U}$ is smooth. We can then apply Theorem~\ref{thm:GeneralBirDecomposition} to the morphism $f : X \to A$ in order to obtain the requisite birational map
\[ X \birat X^{\min} \birat X^{\min}_U \birat (B' \times Z)/G \]
where these objects are the same as in Theorem~\ref{thm:GeneralBirDecomposition}. Furthermore, the dimension of $B'$ is equal to the relative dimension of the Iitaka fibration, which is $\dim X - \kappa(X)$. Finally, to observe that $Z$ is of general type, one can use the birational equivalence and the fact that the quotient map by $G$ is a finite \etale covering to conclude that $\kappa(Z) = \kappa(X)$ and $\dim Z = \dim X - \dim B' = \kappa(X)$.
\end{proof}

Now we are ready to prove the main results from the introduction.

\begin{proof}[Proof of Theorem~\ref{thm:KodairaDimFiniteCoverProduct}]
Let $\phi : X \dashrightarrow S$ be the Iitaka fibration. As in the first paragraph of the proof of Theorem~\ref{thm:FiniteCoverProduct}, we may pass to $X^{\min} \rightarrow S$ and assume that the general fiber is a smooth abelian variety. Any smooth fiber is minimal by adjunction, and Lemma \ref{non_contraction} tells us that the general fibers of $\phi : X^{\min} \to S$ are not contracted by the induced map $f^{\min} : X^{\min} \to A$. Therefore, we may apply Theorem~\ref{thm:FiniteCoverProduct} to conclude.
\end{proof}

Next, we prove a generalization of Theorem~\ref{thm:main}.

\begin{theorem} \label{thm:factoringthroughAV}
Let $X$ be an $n$-dimensional minimal smooth projective variety with $\kappa(X) = n-g$, and let $f \colon X \rightarrow A$ be a morphism to an abelian variety $A$. Then the following are equivalent:
\begin{enumerate}[label={\upshape(\arabic*)}]
    \item There exist one-forms $\omega_{1}, \ldots, \omega_{g} \in H^{0}(A, \Omega^{1}_{A})$ such that $f^{\ast}\omega_{1}, \ldots, f^{\ast}\omega_{g}$ are PLI;
    \item $X$ admits a smooth morphism $\varphi \colon X \rightarrow B$ to an abelian variety $B$ of dimension $g$, such that $\varphi$ fits into the commutative diagram
    \begin{center}
    \begin{tikzcd}
    X \arrow[d, swap, "\varphi"] \arrow[r, "f"] & A \arrow[ld, "q"] 
    \\
    B
    \end{tikzcd}
    \end{center}
\end{enumerate}
When either of these hold, we get an isomorphism $X \iso (B' \times Z)/G$ compatible with $\varphi$, where $B' \to B$ is an isogeny with kernel $G$ and $Z$ is a smooth minimal model of general type with a $G$-action and $G$ acts diagonally on the product $B'\times Z$.
\end{theorem}

\begin{proof}
(2) $\implies$ (1) is immediate. Conversely, suppose (1) holds for $X$. By Lemma~\ref{non_contraction} and then Theorem~\ref{thm:FiniteCoverProduct} applied to $f : X \to A$, there exists a birational map
\[ X \birat (B' \times Z)/G \]
where $B' \rightarrow B$ is an isogeny with kernel $G$ and $Z$ has general type. By Corollary~\ref{corollary:G-cover}, there is an \'etale $G$-cover $X'$ of $X$ admitting a $G$-equivariant birational map $\alpha\colon X'\to A'\times Z$, where $A'\to A$ is an isogeny with kernel $G$, $Z$ is a smooth $G$-variety of general type, and the $G$-action on $A'\times Z$ is diagonal. Moreover, $\alpha$ descends to a birational map $\overline{\alpha}\colon X\dashrightarrow (A'\times Z)/G$. By \cite{BCHM10} and \cite{Prokhorov21}, we can run a $G$-equivariant MMP with scaling along a $G$-invariant ample divisor to get a minimal model $Z^{\min}$ of $Z$ with a compatible $G$-action. Hence we get a birational map between two minimal models $X\dashrightarrow (A'\times Z^{\min})/G$, which is a composition of flops by \cite{Kawamata08}. In particular, we have an isomorphism between codimension one open subsets of $X$ and $(A'\times Z^{\min})/G$. Let $U$ denote the largest common open subvariety. Therefore the following diagram commutes
    \begin{center}
    \begin{tikzcd}
    X' \arrow[rr, dashed, "\alpha"] \arrow[d, "p_1"] && A'\times Z^{\min} \arrow[d, "p_2"]
    \\
    X \arrow[rr, dashed, "\overline{\alpha}"] && (A'\times Z^{\min})/G 
    \\
    & U \arrow[ul, hook, "i_1"] \arrow[ur, hook, "i_2"']
    \end{tikzcd}
    \end{center}
Note that $X'$ and $A'\times Z^{\min}$ are birational minimal models. Hence $\alpha$ is again a composite of flops and $p_1^{-1}(U)$ is isomorphic to $p_2^{-1}(U)$ via $\alpha$ since $p_i$ are \etale maps. An argument in the proof of \cite[Thm.~5.10]{HS21(1)} (c.f. \cite[Thm.~3.4]{Hao23}) shows that all flops of $A'\times Z^{\min}$ arise from flops of $Z^{\min}$. Hence $\alpha : p_1^{-1}(U) \iso p_2^{-1}(U) = A'\times V$ for some open $V \subset Z^{\min}$, using the fact that $U$ is chosen to be maximal among opens over which $\overline{\alpha}$ is defined. From the description of $p_2$, the $G$-action on $A'\times V$ is diagonal. Hence after any sequence of $G$-equivariant flops $Z^{\min} \birat Z^+$, we get a diagonal action on $X' \iso (A' \times Z^+)$. Since $X'$ is smooth, $Z^+$ must be smooth as well. Finally, smoothness of the map $\varphi : X \to B$ is immediate from commutativity of the diagram
\begin{center}
    \begin{tikzcd}
        X \arrow[rd, "\varphi"'] \arrow[r, "\sim"] & (B' \times Z)/G \arrow[d, "\pi_1"]
        \\
        & B
    \end{tikzcd}
\end{center}
and smoothness of $\pi_1$ which follows from the fact that $G$ acts freely on $B'$.
\end{proof}


\begin{remark}
Note that the isomorphism $X \iso (B' \times Z)/G$ is \textit{not} compatible with any map to $A$. Indeed, there may not even be a map to $A$ since $B'/G = B$ may only be isogenous to an abelian subvariety of $A$. Even if $G$ is trivial, the isomorphism may not be compatible with $f$ and the projection. For example, consider $X = E \times C$ where $E$ is an elliptic curve and $C$ is a genus $2$ curve with Jacobian $E \times E'$. Mapping to the Albanese $E \times E \times E'$, our construction gives the identity $\id : E \times C$. However, the map to the Albanese does not factor through the first projection $\mathrm{pr}_1 : X \to E$. 
\end{remark}

\begin{proof}[Proof of Theorem~\ref{thm:main}]
By applying Theorem~\ref{thm:factoringthroughAV} in the case that $f$ is the Albanese morphism and $A = \Alb_X$, we recover Theorem~\ref{thm:main}.
\end{proof}

\begin{example}
We cannot expect to have an \textit{isomorphism} $X \cong (B' \times Z)/G$ as opposed to a birational map in the conclusion of Theorem~\ref{thm:factoringthroughAV} since this implies that $X \to B$ is isotrivial. The following example demonstrates some complexity that can appear in the non-minimal case. Consider an embedding $\iota \colon E \hookrightarrow S$ of an elliptic curve $E$ into a smooth surface $S$ of general type. Now take the graph $\Gamma_{\iota} \subset E \times S$ and blow it up to obtain a threefold $X = \mathrm{Bl}_{\Gamma_{\iota}}(E \times X)$, which admits a smooth morphism to $E$. Indeed, the projection $\pi : X \to E$ is smooth because the fiber over $x \in E$ is the blowup of $S$ at the corresponding point $\iota(X)$ which is smooth. However, the fibers are not minimal; while they are birationally equivalent, they are not isomorphic. Thus, $\pi$ is not isotrivial, only birationally isotrivial. 
\end{example}

\begin{proof}[Proof of Theorem \ref{thm:smooth_map_to_simpleAV}]
(2) $\implies$ (1) is immediate. Now suppose (1) holds. In the notation of Lemma \ref{non_contraction}, we consider the map $F \to \Alb_F \to A$. Since $A$ is simple, this map must be either trivial or surjective. However, since $f^* \omega$ is nonvanishing, it cannot arise from the pullback of $\Alb_X \to Q_X$ by \cite[Thm.~2.1]{PS14}. Therefore, the map $\Alb_F \to \Alb_X \to A$ is nonzero so $F \to A$ is surjective. Via the assumption,
\[ \kappa(X) \ge \dim{X} - \dim{A} \]
we see that $\dim{F} \le \dim{A}$ and hence $F \to A$ is generically finite and the above inequalities are equalities. By the same argument as in Theorem~\ref{non_contraction}(i), the fibers $F$ are birational to abelian varieties. Following the proof of Theorem~\ref{thm:factoringthroughAV}, $X$ admits an etale cover $\tau \colon X' \to X$ with $X' \cong Z \times B$, where $B$ is an abelian variety and $Z$ is a variety of general type. By assumption, $\tau^{\ast} f^{\ast} \omega$ is nowhere vanishing. Consider the decomposition
\[ \tau^* f^* \omega = \pi_1^* \omega_Z + \pi_2^* \omega_B \]
if $\omega_B = 0$ then the nontriviality of the zero locus of $\omega_Z$, by \cite[Thm.~1]{PS14}, would contradict  the assumption that $f^* \omega$ is nowhere vanishing. Therefore, $\omega_B \neq 0$ which implies that the map $B \to A$ is nontrivial and hence smooth since $A$ is simple. Pulling back a basis of forms from $A$ along $\{ z \} \times B \rightarrow Z \times B \rightarrow X$ gives an independent collection of forms on $\{ z \} \times B$ for each $z \in Z$. Hence, those same forms pulled back to $X$ are PLI and the map $f$ is smooth. Finally, the Kodaira dimension equality and isotriviality follow from the \etale cover.
\end{proof}

\bibliographystyle{alpha}
\bibliography{refs.bib}

\footnotesize{
\textsc{Department of Mathematics, Columbia University, New York 10027} \\
\indent \textit{E-mail address:} \href{mailto:nathanchen@math.columbia.edu}{nathanchen@math.columbia.edu}

\textsc{Department of Mathematics, Stanford University, California 94305} \\
\indent \textit{E-mail address:} \href{mailto:bvchurch@stanford.edu}{bvchurch@stanford.edu}

\textsc{Department of Mathematics, KU Leuven, Belgium} \\
\indent \textit{E-mail address:} \href{mailto:feng.hao@kuleuven.be}{feng.hao@kuleuven.be}
}

\end{document}